\documentclass{amsart}
\usepackage{amsmath}
\usepackage{amssymb} 
\usepackage{amscd} 
\usepackage[dvipdfmx]{graphicx} 
\usepackage{epsfig}
\usepackage[dvips]{color}
\usepackage[all]{xy}  

\newtheorem{theorem}{Theorem}[section]

\newtheorem{claim}[theorem]{Claim}

\newtheorem{remark}[theorem]{Remark}
\newtheorem{question}[theorem]{Question}

\theoremstyle{definition}

\numberwithin{equation}{section}
\numberwithin{figure}{section}
\numberwithin{table}{section}

\renewcommand{\(}{\textup{(}}
\renewcommand{\)}{\textup{)}}

\begin{document}
\baselineskip 12pt

\title{Generalized torsion for hyperbolic $3$--manifold groups with arbitrary large rank}

\author[T.Ito]{Tetsuya Ito}
\address{Department of Mathematics, Kyoto University, Kyoto 606-8502, JAPAN}
\email{tetitoh@math.kyoto-u.ac.jp}
\thanks{The first named author has been partially supported by JSPS KAKENHI Grant Number JP19K03490, 21H04428.}

\author[K. Motegi]{Kimihiko Motegi}
\address{Department of Mathematics, Nihon University, 
3-25-40 Sakurajosui, Setagaya-ku, 
Tokyo 156--8550, Japan}
\email{motegi.kimihiko@nihon-u.ac.jp}
\thanks{The second named author has been partially supported by JSPS KAKENHI Grant Number JP19K03502, 21H04428, 
and Joint Research Grant of Institute of Natural Sciences at Nihon University for 2021.}

\author[M. Teragaito]{Masakazu Teragaito}
\address{Department of Mathematics and Mathematics Education, Hiroshima University, 
1-1-1 Kagamiyama, Higashi-Hiroshima, 739--8524, Japan}
\email{teragai@hiroshima-u.ac.jp}
\thanks{The third named author has been partially supported by JSPS KAKENHI Grant Number JP20K03587.}

\begin{abstract}
Let $G$ be a group and $g$ a non-trivial element in $G$. 
If some non-empty finite product of conjugates of $g$ equals to the trivial element,
then $g$ is called a \textit{generalized torsion element}. 
To the best of our knowledge, 
we have no hyperbolic $3$--manifold groups with generalized torsion elements whose rank is explicitly known to be greater than two.  
The aim of this short note is to demonstrate that for a given integer $n > 1$ there are infinitely many closed hyperbolic $3$--manifolds $M_n$ which enjoy the property: 
(i) the Heegaard genus of $M_n$ is $n$, 
(ii) the rank of $\pi_1(M_n)$ is $n$, 
and (ii) $\pi_1(M_n)$ has a generalized torsion element.  
Furthermore, we may choose $M_n$ as homology lens spaces and so that the order of the generalized torsion element is arbitrarily large.
\end{abstract}

\maketitle

{
\renewcommand{\thefootnote}{}
\footnotetext{
}

\section{Introduction}
\label{Introduction}
Let $G$ be a group and let $g$ be a non-trivial element in $G$. 
If some non-empty finite product of conjugates of $g$ equals to the identity, 
then $g$ is called a \textit{generalized torsion element}.
In particular, any non-trivial torsion element is a generalized torsion element. 
As a natural generalization of the order of torsion element, 
the \emph{order} of the generalized torsion element $g$ is defined as 
\[
\min\{ k \geq 2 \: | \: \exists x_{1},\ldots,x_k \in G \mbox{ s.t. } (x_1 g x_1^{-1})(x_2 g x_2^{-1}) \cdots(x_{k}g x_{k}^{-1})=1\}. 
\]

A group $G$ is said to be \textit{bi-orderable\/} if $G$ admits
a strict total ordering $<$ which is invariant under multiplication from
the left and right.
That is, if $g<h$, then $agb<ahb$ for any $g,h,a,b\in G$.
In this paper, the trivial group $\{1\}$ is considered to be bi-orderable.

It is easy to see that a bi-orderable group does not have a generalized torsion element. 
Thus the existence of generalized torsion element is an obstruction for a group to be bi-orderable.
It is known that the converse does not hold in general \cite[Chapter 4]{MR}. 
However, for $3$--manifold groups (fundamental groups of $3$--manifolds), 
one may expect that the converse does hold, 
and in \cite{MT_generalized_torsion} the authors proposed the conjecture: 
a $3$--manifold group is bi-orderable if and only if it has no generalized torsion element. 

The conjecture holds for Seifert fiber spaces and Sol manifolds \cite{MT_generalized_torsion}. 
In particular, 
there is a Seifert fiber space $M$ such that 
$\pi_1(M)$ has a generalized torsion element and the rank of  $\pi_1(M)$ is arbitrarily large. 

Very recently Sekino \cite{S} proves that the conjecture holds for once punctured torus bundles. 
In particular, he gives tunnel number two once punctured torus bundles whose fundamental group has a generalized torsion element.

Furthermore, 
since taking connected sum or gluing along incompressible tori preserves generalized torsion, 
we may also construct non-Seifert fibered $3$--manifolds $M$ so that  
$\pi_1(M)$ has a generalized torsion element and arbitrarily large rank. 
See also \cite{IMT_PAMS} for a behavior of generalized torsion elements under 
prime decomposition and torus decomposition. 

Let us restrict our attention to hyperbolic $3$--manifold groups. 
Two-generator, one-relator hyperbolic $3$--manifold groups with generalized torsion elements are given in \cite{NR,Te,Tera_link,MT_generalized_torsion_high_genus}. 
On the other hand, 
to the best of our knowledge, 
we have no hyperbolic $3$--manifold groups with generalized torsion elements 
whose rank is explicitly known to be greater than two. 

From a viewpoint of combinatorial group theory, it seems to be difficult to handle groups with more than two generators. 
The aim of this short note is to demonstrate the following  
by using a geometric interpretation of generalized torsion elements. 

\begin{theorem}
\label{g-torsion_larga_rank}
For a given integer $n > 1$ there are infinitely many closed hyperbolic $3$--manifolds 
\(precisely homology lens spaces\) $M_{n}$ which enjoy the property: 
 \({\rm i}\)\,the Heegaard genus of $M_{n}$ is $n$, 
 \({\rm ii}\)\,the rank of $\pi_1(M_{n})$ is $n$, 
and  \({\rm iii}\)\,$\pi_1(M_{n})$ has a generalized torsion element. 
Furthermore, we may choose $M_n$ so that the order of the generalized torsion element is arbitrarily large. 
\end{theorem}

\section{Singular spanning disk and generalized torsion}

We recall a useful construction to produce a generalized torsion element. 
A \emph{singular spanning disk} of a knot $K$ in $S^{3}$ is a smooth map 
$\Phi:D^{2} \rightarrow S^{3}$ (or, its image) such that $\Phi(\partial D) = K$ and that $K$ intersects $\Phi(\mathrm{int}D)$ transversely in finitely many points. 
Each intersection point $\Phi(\mathrm{int}D \cap K)$ has a sign according to the orientations. 
We call $\Phi(D)$ a $(p, q)$--\emph{singular spanning disk} if 
$K$ intersects $\Phi(\mathrm{int}D)$ positively in $p$ points and negatively in $q$ points.

For a knot $K \subset S^3$, 
we denote the $3$--manifold obtained by $p/q$--surgery on $K$ by $K(p/q)$. 
Without loss of generality, we may assume $p \ge 0$. 

We recall the following result \cite{IMT_filling}. 

\begin{theorem}
\label{theorem:g_torsion_span_disk}
Let $K$ be a knot in $S^3$. 
If $K$ has a $(n, 0)$--singular spanning disk, 
then the image of a meridian $\mu$ in $\pi_1(K(p\slash q))$ is a generalized torsion element of order $p$ whenever $\frac{p}{q} \ge n$. 
Similarly, if $K$ has a $(0, n)$--singular spanning disk,  
then the image of a meridian $\mu$ in $\pi_1(K(p\slash q))$  is a generalized torsion element  of order $p$ whenever $\frac{p}{q}\le -n$. 
\end{theorem}

A singular disk $D$ which has only clasp singularity as illustrated in Figure~\ref{clasp} is said to be 
\emph{coherent} if $K$ intersects $\mathrm{int}D$ in the same direction. 
A typical and fundamental example of $(n, 0)$-- or $(0, n)$--singular spanning disk is a \emph{coherent clasp disk}. 

\begin{figure}[h]
\includegraphics*[width=0.23\textwidth]{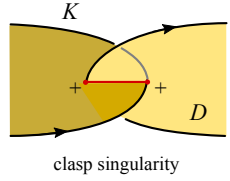}
\caption{Clasp disk}
\label{clasp}
\end{figure}

\medskip

\section{Proof of Theorem~\ref{g-torsion_larga_rank}} 

We denote the Heegaard genus of a $3$--manifold $M$ by $g(M)$. 
We prove Theorem~\ref{g-torsion_larga_rank}} by giving a family of concrete examples below. 

%

Let $K_n$ be a Montesinos knot with rational tangles $R_1, \ldots, R_{n-1}, R_n$, 
where $R_i = [2, -2, 2]$ for $i = 1, \ldots, n-1$ and $R_n = [-6, 4]$.
Then $K_n$ is a Montesinos knot 
\[
M\left(\frac{\beta_1}{\alpha_1}, \ldots, \frac{\beta_{n-1}}{\alpha_{n-1}}, \frac{\beta_n}{\alpha_n}\right),\ 
\textrm{where}\ \frac{\beta_i}{\alpha_i} = \frac{5}{3}\ (1 \le i \le n-1),\ \textrm{and}\ \frac{\beta_n}{\alpha_n} = \frac{17}{6}.
\]
\begin{figure}[h]
\includegraphics*[width=0.6\textwidth]{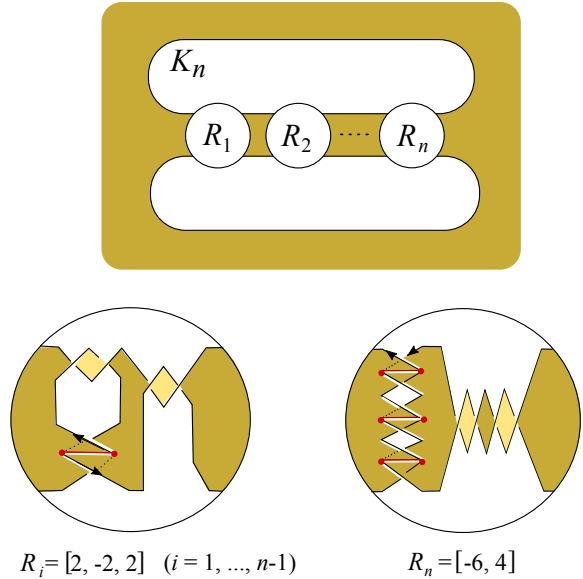}
\caption{$(2n + 4, 0)$--singular spanning disk for $K_n$}
\label{Montesinos_knot_g-torsion}
\end{figure}

We collect some desired properties of $K_n$ as several claims. 

\begin{claim}
\label{tunnel}
The tunnel number of $K_n$ is $n-1$, i.e.~$g(E(K_n)) = n$.
\end{claim}

\begin{proof}
Since $\mathrm{gcd}(\alpha_1, \ldots, \alpha_n) = 3 \ne 1$, 
Theorem~0.1(1) in \cite{LM} shows that the tunnel number of $K_n$ is $n-1$,
i.e.~$g(E(K_n)) = n$. 
\end{proof}

\begin{claim}
\label{rank}
Assume that $p$ is even. 
Then $g(K_n(p/q)) = \mathrm{rank}\, \pi_1(K_n(p/q)) = n$. 
\end{claim}

\begin{proof}
Since $p$ is even and $\mathrm{gcd}(\alpha_1, \ldots, \alpha_n) = 3 \ne 1$, 
we apply \cite[Theorem~0.1(3)]{LM} to obtain the desired result. 
\end{proof}

\begin{claim}
\label{g-torsion}
$\pi_1(K_n(p/q))$ has a generalized torsion element of order $p$ whenever $p/q \ge 2n-4$.
\end{claim}

\begin{proof}
As shown in Figure~\ref{Montesinos_knot_g-torsion}, 
we have a singular spanning disk $\Phi \colon D \to S^3$ such that 
$\Phi(\partial D) = K_n$ and $K_n$ intersects $\Phi(\mathrm{int}D)$ positively in 
$2(n-1) + 6 = 2n + 4$ points, i.e.~ $K_n$ bounds a $(2n+4, 0)$--singular spanning disk. 
Hence, it follows from Theorem~\ref{theorem:g_torsion_span_disk} that 
$\pi_1(K_n(p/q))$ has a generalized torsion element of order $p$ whenever $p/q \ge 2n+4$. 
\end{proof}

\begin{claim}
\label{hyperbolic}
$K_n$ is a hyperbolic knot. 
\end{claim}

\begin{proof}
Following  \cite{Oertel}, 
$K_n$ is either a hyperbolic knot or a torus knot. 
If $n \ge 3$, then the tunnel number of $K_n$ is $n - 1 \ge 2$ (Claim~\ref{tunnel}). 
Since any nontrivial torus knot has tunnel number one \cite{BRZ}, $K_n$ is a hyperbolic knot. 

When $n = 2$, the tunnel number of $K_n$ is one. 
However, it is easy to see that $K_n$ is a two-bridge knot 
$C(2, -2, 6, -7, 1)$ in the Conway notation. 
By the classification of two-bridge knots, 
$K_n$ cannot be a two-bridge torus knot $T_{2, q}$ for any odd integer $q$. 
(See also \cite{BS}.)
\end{proof}

\begin{claim}
\label{surgery_hyperbolic}
$K_n(p/q)$ is hyperbolic for all but finitely many $p/q \in \mathbb{Q}$.
In particular, if $n \ge 4$, 
then $K_n(p/q)$ is hyperbolic for all $p/q \in \mathbb{Q}$. 
\end{claim}

\begin{proof}
Since $K_n$ is hyperbolic (Claim~\ref{hyperbolic}), 
the first assertion follows from Thurston's hyperbolic Dehn surgery theorem \cite{T1,T2}.  
The second assertion follows from \cite{Wu}. 
\end{proof}

It follows from Claims~\ref{tunnel}, \ref{rank}, \ref{g-torsion} and \ref{surgery_hyperbolic} that 
for a given integer $n > 1$, 
we may choose an even integer $p$ so that 
a homology lens space $M_n = K_n(p/q)$ enjoys the required property. 

This completes a proof of Theorem~\ref{g-torsion_larga_rank}.
\hfill $\Box$

\medskip

\begin{remark}
Since $\pi_1(K_n(p/q))$ has a generalized torsion element, it is not bi-orderable.
\end{remark}

Currently, when we restrict to hyperbolic case, 
our technique of constructing a generalized torsion element based on singular clasp disk only can produce a non-integral homology sphere; a generalized torsion obtained by this method is a nontrivial torsion element of its homology.

\begin{question}
\label{question:hyperbolic_integral_sphere}
Does there exist a hyperbolic 3-manifold with torsion-free $1$st homology group 
whose fundamental group has arbitrarily large rank and 
generalized torsion? 
In particular, does there exist a hyperbolic integral homology sphere whose fundamental group has arbitrarily large rank and 
generalized torsion?
\end{question}

Recall that in our current knowledge, 
as for a bi-orderability of a knot group, the Alexander polynomial plays a fundamental role; 
when $K$ is fibered (more generally, when $[G(K),G(K)]$ is residually torsion-free-nilpotent) and all the roots of its Alexander polynomial $\Delta_K(t)$ is positive real, 
then $G(K)$ is bi-orderable, hence in particular, $G(K)$ has no generalized torsion. 

Conversely, when $K$ is rationally homologically fibered, 
i.e. $\deg \Delta_K(t)=2g(K)$,  
and $\Delta_K(t)$ has no positive real roots then $G(K)$ is not bi-orderable \cite{Ito_NYJM}. 
To prove non-bi-orderability, we usually construct a generalized torsion element in the meta-abelian quotient $G(K)\slash [[G(K),G(K)],[G(K),G(K)]]$, with a help of the Alexander polynomial, providing a clue to find a generalized torsion element of $G(K)$ itself.
Thus, from a bi-orderablity point of view, the following counterpart of Question \ref{question:hyperbolic_integral_sphere} is interesting.


\begin{question}
Does there exist a hyperbolic knot with trivial Alexander polynomial whose fundamental group has arbitrarily large rank and 
generalized torsion?
\end{question}

\end{document}